\documentclass{amsart}
\usepackage{latexsym}
\usepackage{amsfonts}
\usepackage{amssymb}
\usepackage{epsfig}
\usepackage{rotating}
\usepackage{amsmath}
\usepackage{color}
\input xy
\xyoption{all}

\newtheorem{theorem}{Theorem}[section]
\newtheorem{lemma}[theorem]{Lemma}
\newtheorem{corollary}[theorem]{Corollary}
\newtheorem{proposition}[theorem]{Proposition}

\newtheorem{definition}[theorem]{Definition}
\newtheorem{remark}[theorem]{Remark}

\newcommand{\bbR}{\mathbb{R}}

\newcommand{\bbZ}{\mathbb{Z}}

\newcommand{\bbC}{\mathbb{C}}
\newcommand{\abs}[1]{\left| #1\right|}

\newcommand{\leqs}{\leqslant}
\newcommand{\geqs}{\geqslant}

\newcommand{\maps}{\longrightarrow}

\newcommand{\homeo}{\cong}
\newcommand{\isom}{\cong}
\newcommand{\cross}{\times}
\newcommand{\C}{\mathcal{C}}

\newcommand{\injects}{\hookrightarrow}

\newcommand{\G}{\mathcal{G}}

\newcommand{\D}{\mathcal{D}}

\def\colon{{:}\;}

\newcommand{\Hom}{\mathrm{Hom}}

\newcommand{\mO}{\mathcal{O}}

\newcommand{\A}{\mathcal{A}}

\newcommand{\srm}[1]{\stackrel{#1}{\maps}}

\newcommand{\End}{\mathrm{End}}

\newcommand{\cA}{\mathcal{A}}

\newcommand{\cG}{\mathcal{G}}

\newcommand{\fk}{\mathfrak{k}}

\newcommand{\Ad}{\mathrm{Ad}}

\newcommand{\wt}[1]{\widetilde{#1}}
\newcommand{\Isom}{\mathrm{Isom}}

\newcommand{\goesto}{\mapsto}
\newcommand{\Hilb}{{\bf H}}

\newcommand{\fu}{\mathfrak{u}}

\def\co{\colon\thinspace}
\begin{document}

\title[Invariant tubular neighborhoods and Yang--Mills theory]{Invariant tubular neighborhoods in infinite-dimensional Riemannian geometry, with applications to Yang--Mills theory}

\author{Daniel Ramras}
\address{Department of Mathematical Sciences,
New Mexico State University, Las Cruces, New Mexico 88003-8001, USA}
\email{ramras@nmsu.edu}

\thanks{This work was partially supported by
    NSF grants DMS-0353640 (RTG), DMS-0804553, and DMS-0968766.}


\begin{abstract}
We show that if $G$ is an arbitrary group acting isometrically on an a (possibly infinite dimensional) Riemannian manifold, then every $G$--invariant submanifold with locally trivial normal bundle admits a $G$--invariant total tubular neighborhood.  These results apply, in particular, to the Morse strata of the Yang-Mills functional over a closed surface.  The resulting neighborhoods play an important role in calculations of gauge-equivariant cohomology for moduli spaces of flat connections over non-orientable surfaces (Baird \cite{Baird-rank3}, Ho--Liu \cite{Ho-Liu-anti-perfect}, Ho--Liu--Ramras \cite{Ho-Liu-Ramras}).
\end{abstract}  

\maketitle{}

\section{Introduction}
Consider a (possibly infinite-dimensional) smooth manifold $X$ acted on by a group $G$, and a smooth, $G$--invariant submanifold $Y\subset X$, which need not be closed.  The question we address is: does $Y$ have a $G$--invariant tubular neighborhood inside $X$?  For finite-dimensional manifolds and compact groups, it is well-known that the answer is yes (see Bredon \cite{Bredon-transf}, for example).  The proof for compact groups involves averaging over the group, and the problem appears to be much more difficult for non-compact groups.  A number of authors have considered specific infinite-dimensional versions of this problem: Ebin \cite{Ebin} constructed invariant tubular neighborhoods for orbits of the diffeomorphism group acting on the space of Riemannian metrics on a manifold; Kondracki and Rogulski \cite{KR} handled the case of orbits of the gauge group acting on the space of connections on a principal bundle; and recently A. Teleman \cite{Teleman-harmonic} gave an explicit construction
of a gauge-invariant tubular neighborhood for the space of connections $A$ on a Euclidean bundle satisfying 
$\mathrm{dim} (\ker (d_A)) = 1$. 

We will show that invariant tubular neighborhoods exist whenever $X$ is a Riemannian manifold, $G$ acts by isometries, and the orthogonal complement of $TY$ in $TX|_Y$ is a smooth subbundle (Proposition~\ref{tubular} and Theorem \ref{total-tubular}).  This last condition is automatically satisfied in the finite dimensional setting, but it is important to note that it does not come for free in the infinite-dimensional setting.  
Our motivation for this work comes from Yang--Mills theory, and we will show that these conditions are all satisfied in the case of the Morse strata of the Yang-Mills functional over a (possibly non-orientable) surface. 

Mathematical applications of Yang--Mills theory began with Atiyah and Bott \cite{A-B}, who 
established a recursive formula for the gauge-equivariant cohomology of
the space of semi-stable holomorphic structures on a Hermitian bundle $E$ over a
Riemann surface.  They achieved this by treating the Yang--Mills functional as a gauge-equivariant Morse function on the infinite-dimensional space $\A(E)$ of connections on $E$.  
In ordinary Morse theory on a manifold $M$, each critical point corresponds to a cell in $M$, whose dimension is the index of that critical point.  The zero-dimensional (co)homology of the critical point  contributes to the (co)homology of $M$, with a shift in dimension corresponding to the index.  Atiyah and Bott showed that the Morse strata of the Yang--Mills functional can be ordered, starting with the central (or semi-stable) stratum, so that each initial segment of the ordering forms an open set in $\A(E)$ (see Ramras \cite{Ramras-YM} for a detailed discussion).  Adding the strata $C_1, C_2, \ldots$ in order, the Thom isomorphism theorem shows that each stratum contributes its cohomology with a shift in dimension, given by the codimension of the stratum: letting $\C_n = \cup_{i = 1}^n C_i$, we have a long exact sequence
\begin{equation}\label{Thom}
 \cdots \to H^*_{\G(E)} (C_n, C_{n-1}) \isom H^{*-\mathrm{codim} (\C_{n})}_{\G(E)} \left(\C_{n}\right) \to H^*_{\G(E)} (\C_n)
\to H^*_{\G(E)} (\C_{n-1}) \to \cdots
\end{equation}
As discussed in Section~\ref{YM}, the $\G(E)$--invariant tubular neighborhoods constructed in this paper provide one method for establishing these Thom isomorphisms.  (Atiyah and Bott proved that the Yang--Mills functional is ``equivariantly perfect," meaning that the boundary maps in the sequence (\ref{Thom}) are all zero.  This leads to the recursive formula mentioned above.)

Ho and Liu \cite{Ho-Liu-non-orient} extended Yang--Mills theory to bundles $E$ over non-orientable surfaces, where the Yang--Mills
stratification is no longer ``equivariantly perfect." 
However, the Thom isomorphisms (\ref{Thom}), together 
with the orientability result of Ho, Liu, and Ramras \cite{Ho-Liu-Ramras},
establish equivariant Morse \emph{inequalities} for the space of flat connections on $E$ (see the introduction to \cite{Ho-Liu-Ramras}). 
Moreover, these Thom isomorphisms
have been used by Tom Baird \cite{Baird-rank3} to produce
formulas for the $U(3)$--equivariant Poincar\'{e} series of $\Hom(\pi_1
\Sigma, U(3))$ (where $\Sigma$ is a non-orientable surface),
establishing the ``anti-perfection'' conjecture of Ho and Liu \cite{Ho-Liu-anti-perfect}.

We note that over a Riemann surface, an alternate approach to the Thom isomorphism in (\ref{Thom}) would be to use the fact that after modding out the (based, complex) gauge group, one obtains a smooth algebraic variety, and results of Shatz \cite[Section 4]{Shatz} show that the image of each stratum is a smooth subvariety.  This algebraic approach involves various technicalities, which we will not attempt to resolve here.  

This paper is organized as follows.  In Section~\ref{inv}, we present our construction of invariant tubular neighborhoods, and in Section~\ref{YM} we use results of Daskalopoulos to show that our construction applies to the Morse strata of the Yang--Mills functional over a surface.  In the last section, we compare our construction of tubular neighborhoods to other constructions in the literature.

\vspace{.1in}

\noindent {\bf Acknowledgements:} I thank Robert Lipshitz and Chiu-Chu Melissa Liu for helpful conversations.  In particular, the construction of tubular neighborhoods presented here owes much to conversations with Lipshitz and Liu.  I also thank Chris Woodward for a helpful discussion regarding the connection with the work of Shatz.


\newpage

\section{Invariant Tubular Neighborhoods}$\label{inv}$

For basic definitions and terminology regarding infinite dimensional manifolds, as well as for many of the necessary constructions in this section, we follow Lang \cite{Lang-dg}.  By smooth we will always mean $C^{\infty}$.

Throughout this section, $X$ will denote a (possibly infinite-dimensional) Riemannian manifold, and $G$ will denote a group acting on $X$ by smooth isometries.  In more detail, this means $X$ is modeled on a Hilbert space $\Hilb$ (real or complex), and the tangent bundle $T(X)$ has a smoothly varying fiber-wise inner product (which we assume to be Hermitian in the complex case).  We will assume these inner products induce the correct topology on $T_x (X)$, or in other words, we assume that the Riemannian metric on $X$ is \emph{strong}.  We will view the action of $G$ as a homomorphism $G\to \Isom(X)$, where $\Isom(X)$ denotes the group of smooth isometries.  For the purposes of our arguments, both $G$ and $\Isom(X)$ can be viewed as \emph{discrete} groups.

Let $Y\subset X$ be a (locally closed), $G$--invariant submanifold; then the tangent space $T_y (Y)$ is a closed subspace of the Hilbert space
$T_y (X)$ and hence has an orthogonal complement $N_y (Y)$.  The normal bundle $N(Y)\subset T(X)|_Y$ is then a well-defined subset of $T(X)|_Y$, and since $G$ preserves the metric, it acts on $N(Y)$.  However, $N(Y)$ need not be a smooth subbundle in general.  For the construction of tubular neighborhoods given by Lang \cite[IV.5]{Lang-dg}, this is irrelevant: Lang obtains a tubular neighborhood diffeomorphic to the quotient bundle $T(X)|_Y/T(Y)$ by using a splitting of the projection $T(X)|_Y/T(Y)$, which always exists if $X$ admits partitions of unity \cite[III, Proposition 5.2]{Lang-dg}.

In this section, we show that if $N(Y)$ is a smooth subbundle of $T(X)|_Y$, then $Y$ admits a $G$--invariant total tubular neighborhood $\tau(Y)$ (Theorem~\ref{total-tubular}).  The neighborhoods $\tau(Y)$ are homeomorphic to the normal bundle $N(Y)$, but may not be diffeomorphic to $N(Y)$.  Since we are mainly interested in Thom isomorphisms, a homeomorphism is sufficient.  (If $G = \{1\}$, or if $Y$ consists of a single orbit of $G$, then we do obtain a diffeomorphism $\tau(Y) \homeo N(Y)$.)  

\begin{remark} The condition that $N(Y)\subset T(X)|_Y$ is a smooth subbundle could be replaced by the condition that there exists a splitting $s$ of the map $\pi$ in the exact sequence
$$0 \maps T(Y) \maps T(X)|_Y \srm{\pi} T(X)|_Y/T(Y)\maps 0$$
whose image in $T(X)|_Y$ is invariant under $G$; our proof goes through equally well in this case, with $\mathrm{Im} (s)$ playing the role of $N(Y)$.
\end{remark}

\subsection{The tangent bundle as a Riemannian manifold}

Associated to the metric on $X$ we have the metric spray $F: TX \to T(TX)$, which is equivariant with respect to the natural actions of the isometry group $\Isom(X)$.  The spray $F$ induces an exponential map $f\co \D \to X$ (Lang \cite[VII.7]{Lang-dg}), whose domain $\D$ is an open subset of $T(X)$ containing the zero section.  It follows from the construction of $f$ that $\D$ is invariant under $\Isom(X)$ and $f$ is $\Isom(X)$--equivariant.

In addition to the exponential map $f$, we will also need a $\Isom(X)$--invariant Riemannian metric on the tangent bundle $TX$, considered as a smooth manifold in its own right.  This Riemannian metric on $T(TX)$ can be defined by explicit local formulae as in Do Carmo \cite[Chapter 3, Exercise 2]{DoCarmo}.  Here we give a more global perspective, from which it is clear that the action of $\Isom(X)$ on $T(TX)$ preserves this metric.

Recall that there is a natural exact sequence of vector bundles over $TX$ of the form
\begin{equation} \label{exact-seq} 0 \maps \pi_{TX}^* (TX) \maps T(TX) \srm{\alpha} (\pi_{TX})^* TX\maps 0,
\end{equation}
where $\pi_{TX} \co TX \to X$ is the structure map for the tangent bundle $TX$, and $\alpha$ is defined via the commutative diagram
$$\xymatrix{ T(TX) \ar@/^/[rrd]^{D \pi_{TX}} \ar@/_/[ddr]_{\pi_{T(TX)}} \ar@{.>}[dr]^\alpha\\
		& (\pi_{TX})^* TX \ar[r] \ar[d] & TX \ar[d]^{\pi_{TX}} \\
		& TX \ar[r]^{\pi_{TX}} & X.
		}
$$
The map $\alpha$ is surjective, because in a local chart $U \subset X$, this diagram becomes (see Lang  \cite[p. 96]{Lang-dg})
$$\xymatrix{ (U\cross \Hilb) \cross (\Hilb\cross \Hilb) \ar@/^/[rrrd]^{(\pi_1, \pi_3)} \ar@/_/[ddr]_{(\pi_1, \pi_2)}
					\ar@{.>}[dr]|-{(\pi_1, \pi_2, \pi_3)}\\
		& U \cross \Hilb \cross \Hilb \ar[rr]^{(\pi_1, \pi_3)} \ar[d]^{(\pi_1, \pi_2)} & & U \cross \Hilb \ar[d]^{\pi_1} \\
		& U \cross \Hilb \ar[rr]^{\pi_1} & & U.
		}
$$
To identify the kernel of $\alpha$, consider the map
$$\pi_{TX}^* (TX) = TX \oplus TX \srm{\phi} T(TX)$$
sending a pair of vectors $(v, w)\in T_x X \oplus T_x X$ to the vector in $T_v (T_x X)$ represented by the curve $t\mapsto v+tw$ in $T_x X \subset TX$.  This defines a map of vector bundles over $TX$, and 
from the local representation given above, one sees that $\phi$ is a smooth, injective bundle map whose image is precisely $\ker (\alpha)$.

By naturality, the maps in the exact sequence (\ref{exact-seq}) are $\Isom(X)$--equivariant.  The next step will be to construct an $\Isom(X)$--equivariant splitting of $\alpha$.

Locally, the spray $F$ has the form
$$(x, v) \goesto (x, v, v, f_x (v))$$
where $f_x \co T_x (X)\to \Hilb$ is homogeneous of degree two, meaning that $f_x (tv) = t^2 f_x (tv)$ for any $v\in T_x (X)$ and any scalar $t$ \cite[IV, Proposition 3.2]{Lang-dg}. 
We now recall a well-known lemma (see, for example, \cite[I.3]{Lang-dg}), whose proof we include for completeness.

\begin{lemma} If $q\co E\to F$ is a smooth map between Banach spaces and $q$ is homogeneous of degree $p>0$ (meaning that $q(tx) = t^p q(x)$ for all $x\in E$ and all scalars $t$), then $q(x) = \frac{1}{p!} \left(D^p q(0)\right)(x,x,\ldots, x)$, where $D^p q$ is the $p^\textrm{th}$derivative of $q$, considered as a map from $E$ in the space of multilinear functions $E^p \to F$.  
\end{lemma}
\begin{proof}  The proof is by induction on $p$. 
Differentiating the equation $q(tx) = t^p q(x)$ with respect to $t$ gives 
\begin{equation}\label{diff} \left(Dq(tx)\right) (x) = pt^{p-1} q(x).\end{equation}
Note that when $p=1$, the result follows by setting $t=0$, completing the base case.
Setting $t=1$ in (\ref{diff}), we have
\begin{equation} \label{q} q(x) =  \frac{1}{p} \left(Dq(x)\right)(x).  \end{equation}

Now, $Dq \co E\to L(E, F)$, is homogeneous of degree $p-1$: indeed, for any vector $v\in E$ we can differentiate the equation 
$$q(t(x + r v)) = t^p q(x + r v)$$
with respect to $r$, obtaining
$$\left(Dq(t(x+ r v))\right) (tv) = t^p \left(Dq(x+ r v) \right)(v);$$
setting $r = 0$ and using linearity of $Dq(tx)$, one finds that
$$\left(Dq(tx)\right) (v) = t^{p-1} \left(Dq (x)\right)(v)$$
for any scalar $t$ and any $x,v\in E$.

We can now apply the induction hypothesis to the degree $p-1$ homogeneous function $Dq$, yielding
$Dq(x) =  \frac{1}{(p-1)!} \left(D^p q(0)\right) (x,x,\ldots, x)$.  Plugging this into (\ref{q}) completes the proof.
\end{proof}

From the lemma, it follows that the quadratic function $f_x$ can be written in the form
$$f_x (v) = \frac{1}{2} \left(D^2 f_x (0)\right)(v,v),$$
meaning that $f$ is the quadratic part of the symmetric bilinear form $\frac 1 2 D^2 f_x (0)$.  Writing this bilinear form as $B(x, v, w) = \frac 1 2 \left(D^2 f_x (0)\right) (v,w)$, we obtain a splitting $s\co (\pi_{TX})^* TX = TX\oplus TX \to T(TX)$ of the exact sequence (\ref{exact-seq}) whose local form is 
$$s(x, v, w) = (x, v, w, B(x, v, w)).$$
Since the spray $F$ is $\Isom(X)$--equivariant and (by polarization) $B(x, -, -)$ is the unique symmetric bilinear form satisfying $B(x, v, v) = f_x (v)$, it follows that $s$ is itself $\Isom(X)$--equivariant.  Note also that $s$ is a map of vector bundles over $TX$.

The splitting $s$ gives us an explicit, $\Isom(X)$--equivariant direct sum decomposition of vector bundles over $TX$: 
$$T(TX) \isom \pi_{TX}^* (TX) \oplus \pi_{TX}^* (TX).$$
The Riemannian metric on $X$ pulls back to an inner product on $\pi_{TX}^* (TX)$ (which is again $\Isom(X)$--invariant) and now the direct sum decomposition endows $T(TX)$ with an $\Isom(X)$--invariant inner product as well, making $TX$ into a Riemannian manifold with an isometric action of the group $\Isom(X)$.

We note that there is another description of this splitting $s$ in terms of the Levi-Civita covariant derivative $D$ \cite[VIII, Theorem 4.2]{Lang-dg}: given $(v,w)\in T_x(X) \oplus T_x (X)$ with $w = \gamma'(0)$ for some curve $\gamma\co (-\epsilon, \epsilon) \to X$, there is a unique lift of $\gamma$ to a curve $\gamma_v: (-\epsilon, \epsilon) \to TX$ such that $\gamma(0) = v$ and $\gamma_v$ is parallel to $\gamma$ with respect to $D$.  The splitting $s$ now sends $(v, w)$ to $(\gamma_v)'(0) \in T_v (TX)$.  This splitting is the same as the one given above, as can be seen by examining the definition of $D$--parallel curves.

\subsection{Construction of the tubular neighborhood}

\begin{definition}$\label{tub-def}$  A \emph{tubular neighborhood} of $Y$ in $X$ is an open neighborhood $\tau(Y)$ of $Y$ in $X$ together with a homeomorphism 
$\phi\co U\to \tau(Y)$ from an open neighborhood $U\subset N(Y)$ containing the zero-section $Y\subset N(Y)$, such that $\phi$ restricts to the identity from $Y\subset N(Y)$ to $Y\subset X$.  If $\phi$ is a diffeomorphism, we call $\tau(Y)$ smooth, and if $U = N(X)$, we call $\tau$ \emph{total}.
\end{definition}

\begin{proposition}$\label{tubular}$  Let $X$ be a Riemannian manifold and let $Y\subset X$ be a (locally closed) submanifold whose normal bundle $N(Y)\subset T(X)|_Y$ is a smooth subbundle.  Assume $G$ acts on $X$ by smooth isometries, leaving $Y$ invariant.  Then there exists a $G$--invariant open neighborhood $Z$ of the zero section of $N(Y)$ with $Z\subset \D\cap N(Y)$, and the exponential map restricts to an equivariant diffeomorphism $f\co Z\to V$, where $V$ is an open neighborhood of $Y$ in $X$.  Hence $V$ is a $G$--invariant tubular neighborhood of $Y$.  
\end{proposition}
\begin{proof} We identify $Y$ with the zero section of $N(Y)$, and for $g \in G$ we will let $g$ denote both the self-map it induces on $X$ and the derivative of this map. 

We begin by noting that the geodesic distance $d$ associated to the natural $\Isom$--invariant Riemannian metric on $TX$ constructed above yields an $\Isom(X)$--invariant distance function (topological metric) on $TX$ (see \cite[VII.6]{Lang-dg} for a discussion of the geodesic distance, and the fact that the associated metric topology is the usual topology on $TX$).  Since $G$ acts by isometries, the distance function $d$ is $G$--invariant.
For any $W\subset T(X)$ and any $w\in W$, we write $B_{\epsilon} (w, W) = \{w'\in W| d(w, w')<\epsilon\}$.  Note that we identify $X$ with the zero section of $T(X)$, so $B_\epsilon (x, X)$ is defined.

Now fix $y\in Y$. Since the exponential map $f$ restricts to a local diffeomorphism $\D\cap N(Y)\to X$
\cite[p. 109]{Lang-dg}, we know that for some $\epsilon_y>0$, $f$ restricts to a diffeomorphism $B_{\epsilon_y} (y, N(Y))\to f\left(B_{\epsilon_y} (y, N(Y))\right)$ (with $f\left(B_{\epsilon_y} (y, N(Y))\right)$ \emph{open} in $X$).  
Let $\psi\co f\left(B_{\epsilon_y} (y, N(Y))\right)\to B_{\epsilon_y} (y, N(Y))$ denote the inverse map.  Now, $f\left(B_{\epsilon_y/2} (y, N(Y))\right)$ is an open neighborhood of $y$ in $X$, hence contains $B_{\epsilon'_y}(y, X)$ for some $\epsilon'_y<\epsilon_y/4$.  Set $Z_y = \psi(B_{\epsilon'_y}(y, X))\subset N(Y)$.  Then $Z_y$ is open in $N(Y)$ and we have 
\begin{equation}\label{Z_y0}
Z_y \subset B_{\epsilon_y/2} (y, N(Y))
\end{equation}
and
$f(Z_y) = B_{\epsilon'_y}(y, X)$.  In fact, for any $g\in G$,
\begin{equation}\label{Z_y}
f(g (Z_y)) = B_{\epsilon'_y} (g\cdot y, X).
\end{equation}

Define $Z_{y(\mO)}$ as above for one point $y(\mO)$ from each $G$--orbit $\mO\subset Y$.  Then we claim that $f$ is injective, and hence restricts to a diffeomorphism, on the open set
$$Z = \bigcup_{\mO \in Y/G, \,\, g\in G} g(Z_{y(\mO)}).$$
Note here that in order to conclude that $f|_Z$ is a diffeomorphism onto its image, we are using the fact that $f\left(B_{\epsilon_y} (y, N(Y))\right)$ is \emph{open} in $X$; this also shows that $f(Z)$ is open in $X$.

Say $x = f(g_1 \cdot z_1) = f(g_2 \cdot z_2)$ with $z_1\in Z_{y_1}$, $z_2\in Z_{y_2}$ and $g_1$, $g_2\in G$ (here the $y_i$ are the chosen representatives for some two orbits).  Then by (\ref{Z_y}) 
we have
\begin{equation}\label{est}
d(g_1\cdot y_1, g_2 \cdot y_2) \leq d(g_1\cdot y_1, x) + d(x, g_2\cdot y_2) < \epsilon'_{y_1} + \epsilon'_{y_2} < \epsilon_{y_1}/4 + \epsilon_{y_2}/4.
\end{equation}
We may assume that $\epsilon_{y_1} \geq \epsilon_{y_2}$.  Then by (\ref{Z_y0}) and (\ref{est}) we have
\begin{equation}\label{est2}
\begin{split}
d(g_1 \cdot y_1, g_2 \cdot z_2) &\leqs d(g_1\cdot y_1, g_2\cdot y_2) + d(g_2\cdot y_2, g_2\cdot z_2) \\
&< (\epsilon_{y_1}/4 + \epsilon_{y_2}/4) + \epsilon_{y_2}/2
\leqs \epsilon_{y_1}.
\end{split}
\end{equation}
The fact that $f$ is injective on $B_{\epsilon_{y_1}} (y_1, N(Y))$ implies that it is also injective on $B_{\epsilon_{y_1}} (g_1 \cdot y_1, N(Y))$.  But $g_1 \cdot z_1 \in B_{\epsilon_{y_1}} (g_1 \cdot y_1, N(Y))$ by (\ref{Z_y0}), and $g_2 \cdot z_2 \in B_{\epsilon_{y_1}} (g_1 \cdot y_1, N(Y))$ by (\ref{est2}), so we have a contradiction.
\end{proof}

Before constructing a $G$--invariant total tubular neighborhood, we need a lemma which follows from the argument 
in Lang \cite[VII, Proposition 4.1]{Lang-dg}.

\begin{lemma}$\label{compression}$ Let $X$, $Y$, $Z$, and $G$ be as in Proposition~\ref{tubular}, and let $\sigma\co Y\to\bbR$ be a continuous $G$--invariant function such that for all $y\in Y$,
$\sigma(y) >0$ and 
$\{v\in N_y (Y) \, | \, \abs{v} < \sigma (y)\}$ lies inside $Z$.

Then there is a homeomorphism $N(Y) \to N(Y)_{\sigma}$, where 
$$N(Y)_{\sigma} = \{ v\in N(Y) \,\, |\,\, ||v||<\sigma(\pi v)\}$$
(here $\pi\co N(Y)\to Y$ is the projection, and $||\cdot||$ is the norm on $T_x (X)$).  This homeomorphism restricts to the identity on $Y\subset N(Y)$.
\end{lemma}

We can now prove our main theorem.

\begin{theorem}$\label{total-tubular}$ Let $X$ be a Riemannian manifold, and let $G$ act on $X$ by isometries, leaving the submanifold $Y\subset X$ invariant.  Assume that $N(Y)$ is a smooth subbundle of $T(X)|_Y$.  Then there exists a $G$--invariant total tubular neighborhood $\tau(Y)$ of $Y$ in $X$.  The homeomorphism
$N(Y) \stackrel{\homeo}{\maps}  \tau(Y)$ is $G$--equivariant, but need not be smooth.
\end{theorem}
\begin{proof}  We apply Lemma~\ref{compression}.  Set $\sigma(y) = \sup \{\epsilon>0 : ||v||<\epsilon \implies v\in Z\}$.
Then $\sigma(y)>0$ for any $y\in Y$, $\sigma$ is $G$--invariant, and 
we have $N(Y)_\sigma \subset Z$.    
Continuity of $\sigma$ follows from the assumption that the inner products defining $||\cdot||$ vary continuously over $X$, together with the fact that $Z$ is open.
\end{proof}

\begin{remark} If $G = \{1\}$ and $X$ admits partitions of unity, then Lang \cite[VII.4, Corollary 4.2]{Lang-dg} shows that one may choose $\sigma$ to be smooth, and then the homeomorphism $N(Y) \to \tau(Y)$ will be smooth as well.  If $Y$ consists of a single $G$--orbit, then $\sigma$ can be chosen to be \emph{constant}, and we obtain the same conclusion.
\end{remark}


\section{Applications to Yang--Mills Theory}$\label{YM}$

We now apply Theorem~\ref{total-tubular} to the Morse strata of the Yang--Mills functional.  Let $P\to M$ be a principal $U(n)$--bundle over a surface, and
let $\cA^{k-1}(P)$ denote the affine space of connections on $P$ lying in the Sobolev space 
$L^2_{k-1}$ ($k\geqs 2 1$).  Similarly, let $\cG^k(P)$ denote the $L^2_{k-1}$ Sobolev completion of the unitary gauge group of $P$, which acts smoothly on $\cA^{k-1} (P)$ (Wehrheim \cite{Wehrheim}).  Connections on $P$ correspond precisely to Hermitian connections on the associated Hermitian bundle $P\cross_{U(n)} \bbC^n$, and $\G^k(P)$ is isomorphic to the Sobolev gauge group $\G^k (E)$.  The following lemma will allow us to apply Theorem~\ref{total-tubular}.  A proof of this lemma may be found 
in Kondracki--Rogulski \cite[Section 2.3]{KR}.

\begin{lemma}\label{inv-metric}
Let $K$ be a compact connected Lie group and
let $P$ be a principal $K$--bundle over
a Riemann surface $M$ equipped with 
a Riemannian metric. The metric
on $M$ and a $K$--invariant inner product
on the Lie algebra $\fk$ of $K$ induce
a Riemannian metric on $\cA^{k-1}(P)$,
the space of $L^2_{k-1}$ $K$--connections on $P$. 
This Riemannian metric
is $\cG^k(P)$--invariant.
\end{lemma}

\vspace{.15in}

As discussed in the introduction, the Morse strata of the Yang--Mills functional can be ordered  $C_1\leqs C_2\leqs \cdots$ , starting with the central stratum, in such a way that the union $\C_n = \cup_{i = 1}^n C_i$ of each initial segment of the ordering is open.  Such an ordering was first discussed by Atiyah and Bott \cite{A-B}, and is described in more detail in Ramras \cite{Ramras-YM}.  We can now give our main application of Theorem~\ref{total-tubular}.

\begin{corollary}$\label{YM-tubular}$ Let $P\to M$ be a smooth principal $U(n)$--bundle over a surface $M$, and let $\leqs$ denote a linear order on the set of Morse strata as above.  Then each stratum $S$ has a $\G(E)$--invariant tubular neighborhood lying inside the open set $\C_n$, and there are Thom isomorphisms in gauge equivariant cohomology
$$H^*_{\G(E)} (C_n, C_{n-1}) \isom H^{*-\mathrm{codim} (\C_{n})}_{\G(E)} \left(\C_{n}\right).$$ 
If $M$ is orientable, or a non-orientable surface whose double cover has genus greater than 1, these isomorphisms hold with integer coefficients; in general they hold with $\bbZ/2$--coefficients.
\end{corollary}  

\begin{proof} Since $\cA^{k-1} (P)$ is an affine space modeled on the vector space of $L^2_{k-1}$ sections of $\mathrm{ad} P \otimes T^*M$, this space is a Hilbert manifold, as are the open subsets $\bigcup_{i\leqs j} C_i$.  When $M$ is orientable, the Morse strata for the Yang--Mills functional on $P$ are locally closed submanifolds of $\cA^{k-1} (P)$ (Daskalopoulos \cite{Dask}) and are invariant under $\cG^k(P)$.  So existence of the desired neighborhoods will follow by applying Theorem~\ref{total-tubular} and Lemma~\ref{inv-metric} to the submanifolds $C_j \subset \bigcup_{i\leqs j} C_i$, once we show that the normal bundle $N(C_j)\subset T(\A^{k-1} (P))|_{C_j}$ is a smooth subbundle.  Since the kernel of a smooth, surjective map of vector bundles is a smooth subbundle of the domain (see, for example, Lang \cite[III.3]{Lang-dg}), it will suffice to show that the orthogonal projections $T_A (\A^{k-1} (P)) \to T_A (\C_j)$ vary smoothly with $A\in \C_j$.

To understand these orthogonal projections, we follow Daskalopoulos' proof that the strata are locally closed submanifolds \cite[Proposition 3.5]{Dask}.  Recall \cite[Section 5]{A-B} that connections on $P$ correspond bijectively with holomorphic structures on the associated vector bundle $\Ad\, P = P\cross_{U(n)} \bbC^n$, where we view a holomorphic structure as an operator on sections
$$d'': L^2_{k-1} (\Ad\, P) \to L^2_{k-1} (T^* M'' \otimes \Ad\, P).$$
Here $T^* M''$ is the space of $(0,1)$--forms on the complex manifold $M$.  Sections in the kernel of $d''$ are then \emph{holomorphic}.  We will call $\Ad\, P$, together with a $d''$ operator, a holomorphic bundle, which we denote by $E$.  Note that by adjointness, $d''$ induces a corresponding operator on the endomorphism bundle $\End (\Ad\, P) = \End\, E$.  

Given a critical holomorphic bundle $E\in C_j$ let $S$ denote the orthogonal complement of the space of $L^2_{k-1}$ holomorphic endomorphisms of $E$ inside the space of all smooth endomorphisms (that is, $S$ is the kernel of the $d''$ operator associated to $\End\, E$).  
Daskalopoulos constructs a map
$$f\co S \cross H^1 (M, \End \, E) \maps \A^{k-1} (P)$$
and shows that $f$ is an isomorphism near zero, meaning that 
$f\co U_0 \maps U_E$
is an isomorphism for some neighborhood $U_0$ of $0$ in the Hilbert space $S \cross H^1 (M, \End \, E)$ and some neighborhood $U_E$ of $E$.  Daskalopoulos then shows that 
$$f^{-1}(U_E \cap C_j) = U_0 \cap \left(S \cross H^1 (M, \End' \, E) \right),$$  
so that $f$ provides a chart for $C_j$ near $E$.
Here $\End' \, E$ is the sheaf of holomorphic endomorphisms preserving the Harder-Narasimhan filtration on $E$.  (The notation $\End' \, E$ comes from Atiyah--Bott \cite[Section 7]{A-B}; Daskalopoulos uses the notation $UT(E, *) = \End' \, E$, since these endomorphisms correspond, in a local ordered basis respecting the Harder-Narasimhan filtration, to upper triangular matrices.)

This chart provides an isomorphism between the tangent bundle to $U_E \cap C_j$ and the trivial bundle
$$\left( U_0 \cap \left(S \cross H^1 (M, \End' \, E)\right) \right) \cross (S\cross H^1 (M, \End' \, E)).$$ 
We need to understand the orthogonal projections 
$$S\cross H^1 (M, \End \, E) \maps S\cross H^1 (M, \End' \, E)$$ with respect to the \emph{pullback}, along $f$, of the metric on $\A^{k-1} (P)$ to the various tangent spaces $S\cross H^1 (M, \End \, E)$ (although these tangent spaces are all isomorphic to one another, the pullback metric will vary depending on the behaviour of the map $f$).   Observe that this orthogonal projection is simply the identity on $S$, plus an orthogonal projection $H^1 (M, \End \, E) \to  H^1 (M, \End' \, E)$.  These cohomology groups are finite dimensional, and since the metrics vary smoothly, so do the orthogonal projections. 

 Daskalopoulos produces the manifold structure at a general point $E\in C_j$ (as opposed to critical points) using a gauge transformation sending $E$ into a neighborhood of the critical set on which the manifold structure has already been established.  So local triviality of $N(C_j)$ at arbitrary points of $E$ follows from the result near the critical set.  This completes the proof of local triviality, in the case of an orientable surface.

When $M$ is non-orientable, we let $\wt{P}\to \wt{M}$ denote the pullback of $P$ to the orientable double cover $\wt{M} \to M$.  There is a natural embedding $\cG^k (P)\injects \cG^k(\wt{P})$, making the natural embedding
$\cA^{k-1} (P)\injects \cA^{k-1} (\wt{P})$ equivariant.  In both cases, the image of the embedding is the set of fixed points of an involution $\tau$ arising from the deck transformation on $\wt{M}$ (see Ho--Liu \cite{Ho-Liu-non-orient}).  The Morse strata in $\cA^{k-1} (\wt{P})$ are, by definition, just the intersections of Morse strata in $\cA^{k-1} (\wt{P})$ with $\cA^{k-1} (P)$, and hence are locally closed submanifolds invariant under $\cG^k (P) = \cG^k (\wt{P})^{\tau}$.  By Lemma~\ref{inv-metric}, we obtain a gauge-invariant metric on $\cA^{k-1} (\wt{P})$ from a choice of a metric on $\wt{M}$ and a $U(n)$--invariant inner product on $\fu (n)$.  We will choose our metric on $\wt{M}$ to be invariant under the deck transformation, so that the resulting metric on $\A^{k-1} (\wt{P})$ is invariant under $\tau$.
Now, the $\cG^k (\wt{P})$--invariant metric on $\cA^{k-1}(\wt{P})$ restricts to a $\cG^k (P) = \cG^k (\wt{P})^{\tau}$--invariant metric on $\cA^{k-1} (\wt{P}) = \cA^{k-1}(P)^{\tau}$, and it remains to check that the normal bundles are smooth subbundles $T(\A^{k-1} (P))$.  

Given a Morse stratum $C \subset \cA^{k-1} (P)$, it will suffice to show that the orthogonal projection $Q_C \co T \cA^{k-1} (P)\to TC$ is a smooth bundle map, since the kernel of this projection is precisely the normal bundle to $C$.  By definition, $C = \wt{C} \cap \cA^{k-1} (P)$ for some Morse stratum $\wt{C} \subset \cA^{k-1} (\wt{P})$.  Let $Q_{\wt{C}} \co T \cA^{k-1} (\wt{P}) \to T \wt{C}$ denote the orthogonal projection, which we have shown is a smooth bundle map.  Since the involution $\tau$ is an isometry, it follows that $Q_C$ is simply the restriction of the $Q_{\wt{C}}$ to $TC = (T\wt{C})^\tau \subset T\wt{C}$ (in general, one may check that if $W$ is a closed subspace of a Hilbert space $\Hilb$ and $\tau\co \Hilb \to \Hilb$ is a linear isometry, then the orthogonal projection $\Hilb^\tau \to W^\tau$ is the restriction of the orthogonal projection $\Hilb\to W$).  Hence $Q_C$ is a smooth bundle map as well.

To produce the desired Thom isomorphisms, note that over a Riemann surface, the normal bundles to the Yang--Mills strata are naturally complex vector bundles (Atiyah--Bott \cite[Section 7]{A-B}).  For any $G$--equivariant complex vector bundle $V\to X$, the homotopy orbit bundle $V_{hG}\to X_{hG}$ is still a complex vector bundle, hence orientable.
Letting $\nu_n$ denote a $\G(E)$--invariant tubular neighborhood of $C_n$ inside $\C_n$, we may excise the complement of $(\nu_n)_{h\G(E)}$ in $(\C_n)_{h\G(E)}$ and apply the ordinary Thom isomorphism for the bundle $(\nu_n)_{h\G(E)}\to (\C_n)_{h\G(E)}$.

The same argument works over a non-orientable surface, since it is proven in Ho--Liu--Ramras \cite{Ho-Liu-Ramras} that the normal bundles to the Yang--Mills strata (and the corresponding homotopy orbit bundles) are orientable (real) vector bundles (so long as the double cover of the surface has genus greater than 1).
\end{proof}


\section{Other constructions of tubular neighborhoods}$\label{background}$

\subsection{General constructions}

In order to put our construction in context, we briefly discuss existing general constructions of tubular neighborhoods for for submanifolds $Y\subset X$.  
There are many treatments of this topic in the literature, and the basic tool is always the exponential map, which is a local diffeomorphism from an open neighborhood $\D \subset N(Y)$ containing the zero section $Y\subset N(Y)$ to an open neighborhood of $Y$ in $X$.  
Since the proof of Godement's lemma \cite{Godement} first appeared, most arguments \cite{Bredon-transf, Brocker-Janich, GP, Hirsch, Lang-dg, MO, Mukherjee} rely on Godemont's idea produce a smaller neighborhood $Z\subset \D$ on which the exponential map is injective, hence a diffeomorphism (in \cite{BG, Milnor-collected-III, Spivak-I}, $Y$ is assumed to be compact, and a simplified version of Godement's lemma is used).  After this point, one argues that $N(Y)$ can be compressed into $Z$.  In order to construct invariant neighborhoods for compact groups, Bredon uses integration over the group to make this compression $G$--equivariant.  Surprisingly, we have not found the non-equivariant version of our argument, making use of a distance function on $T(X)$, elsewhere in the literature.

\subsection{Relationship with the work of Kondracki--Rogulski, Ebin, and Teleman}

Kondracki and Rogulski \cite{KR} provided a construction of gauge-invariant tubular neighborhoods of the gauge orbits inside the space $\cA^{k-1} (P)$ of (Sobolev) connections on a smooth principle $K$--bundle $P$ over a compact manifold $M$, where $K$ is a compact Lie group.  Following Ebin \cite{Ebin}, Kondracki and Rogulski utilize a weak gauge-invariant metric on the space of connections, and show that with respect to this weak metric, the orthogonal complement to $T(\G^k (P) \cdot A)$ inside $T(\cA^{k-1} (P))$ is a smooth subbundle (\cite[Lemma 3.3.1]{KR}).  However, their argument also shows that the orthogonal complement with respect to the strong metric is  a smooth subbundle.  Hence Theorem \ref{total-tubular} applies to this situation. 
 
Ebin \cite{Ebin} considers the space of Riemannian metrics on a fixed finite-dimensional manifold $M$, equipped with the action of the diffeomorphism group of $M$.  After passing to appropriate Sobolev completions of these objects, Ebin constructs diffeomorphism-invariant tubular neighborhoods for the orbits of the action.  His arguments make use of a weak Riemannian metric on the space of metrics, and it is only with respect to this weak metric that Ebin proves the normal bundles to orbits are smooth subbundles.  Ebin's arguments are somewhat subtler than those of Kondracki--Rogulski, and they do not directly apply to the strong Riemannian metric.  Hence our results do not immediately apply in this situation.  We note that Ebin uses his results in the Sobolev setting to derive results in the smooth setting.

A. Teleman \cite{Teleman-harmonic} gives an explicit construction
of a gauge-invariant tubular neighborhood for the space of connections $A$ on a Euclidean bundle satisfying the condition
$\mathrm{dim} (\ker (d_A)) = 1$.  Teleman's methods are quite different from ours, and it is unclear whether our results apply to this situation.


\end{document}